\documentclass{article}
\usepackage{mathtools,amssymb,amsthm}
\usepackage{enumerate}
\usepackage{authblk}
\usepackage[colorlinks]{hyperref}
\usepackage{cleveref}

\theoremstyle{plain}
\newtheorem{theorem}{Theorem}[section]
\newtheorem{proposition}[theorem]{Proposition}
\theoremstyle{definition}
\newtheorem{definition}[theorem]{Definition}
\theoremstyle{remark}
\newtheorem{example}[theorem]{Example}
\newtheorem*{remark}{Remark}

\title{Faithful Decomposition of Rationals}

\author[1,2]{Sunben Chiu}
\author[1]{Pingzhi Yuan}
\author[,3]{Hongjian Li\thanks{Corresponding author. Email: lhj@gdufs.edu.cn.}}

\affil[1]{School of Mathematical Sciences, South China Normal University, Guangzhou, 510631, P. R. China}
\affil[2]{Department of Technology, ICBC Software Development Center, Guangzhou 510665, P. R. China}
\affil[3]{School of Mathematics and Statistics, Guangdong University of Foreign Studies, Guangzhou 510006, P. R. China}
\date{}

\begin{document}
\maketitle
\begin{abstract}
If an irreducible fraction $\frac mn>0$ can be decomposed into the sum of several irreducible proper fractions with different denominators, and the positive number smaller than $\frac mn$ in fractional ideal $\frac 1n\mathbb Z$ can not be obtained by replacing some numerator with smaller non-negative integers, then the decomposition is said to be faithful. For $t\in\mathbb Z$, we prove that the length of faithful decomposition of an irreducible fraction $\frac mn$ with $2\le t\le\frac mn<t+1$ is at least $t+2$. In addition, we show a faithful decomposition of rationals consisting only of unit fractions except for one term. And we write $\frac 4n$ as a faithful decomposition with three fractions at most one non-unit fraction.

\textit{Key words}. fraction decomposition, faithful, length, unit fraction, Egyptian fraction.

2020 \textit{Mathematics Subject Classification}. 11D68, 11A05, 11A41, 11A07.
\end{abstract}

\section{Introduction}

\begin{definition}
Let $\frac mn$ be a positive rational. The sum of form
\begin{equation}\label{equ:decomposition}
\frac mn=\frac{a_1}{b_1}+\frac{a_2}{b_2}+\cdots+\frac{a_t}{b_t}
\end{equation}
is called the \textbf{fraction decomposition} of $\frac mn$, or \textbf{decomposition} for short, where $\frac{a_1}{b_1}$, $\frac{a_2}{b_2}$, $\dots$, $\frac{a_t}{b_t}$ are positive rationals and $b_1$, $b_2$, $\dots$, $b_t$ are different positive integers. The number $t$ is called the \textbf{length} of the decomposition.
\end{definition}

A fraction with the numerator $1$ is called a \textbf{unit fraction}. We denote the principal fractional ideal generated by unit fraction $\frac{1}{n}$ as $\frac{1}{n}\mathbb{Z}$.

\begin{definition}
Let $u=\frac mn$ be a positive irreducible fraction. The decomposition \eqref{equ:decomposition} is called \textbf{faithful} if for any $0\le x_i\le a_i$, $v=\sum_{i=1}^t\frac{x_i}{b_i}\notin\frac{1}{n}\mathbb{Z}$ except for $v=u$ and $v=0$.
\end{definition}

In 2003, Croot III \cite{croot2003coloring} solved the famous problem asked by Erd\H{o}s and Graham \cite[P36]{erdos1980old} in 1980. If the integers greater than 1 are partitioned into finitely many subsets, then one of the subsets has a finite subset of itself whose reciprocals sum to 1. In 2020, Yao \cite{yao2020note} proved that the length of the shortest faithful de\-com\-po\-si\-tion of $n\ge2$ is equal to $n+2$, and conjectured that any positive integer $n$ has a faithful decomposition consisting only of any unit fractions. In 2019, Ding \cite{ding2019note} proved a special fraction decomposition of positive integer. For any integer partition $m=m_1+\dots+m_e$, all the faithful decompositions of $m_i$ construct the fraction decomposition of $m$, whose partial sum in $\mathbb Z$ is in 1-1 correspondence with the partial sum of the $\{m_i\}_{i=1}^e$.

A finite sum of unit fractions is called Egyptian fraction. For a given positive integer $m$ and any $n>1$, whether $\frac mn$ can be written as the sum of three unit fractions is an interesting topic. Schinzel \cite{Selecta,Breusch1956,hagedorn2000proof} proved that the conclusion is true in case $m=3$. Erd\H{o}s \cite[P44]{erdos1980old} conjectured that the conclusion is true in case $m=4$. Franceschine \cite{Franceschine} and Mordell \cite[P287]{mordell1969diophantine} proved that Erd\H{o}s' conjecture holds for all $n<10^8$ and $n\not\equiv1$, $11^2$, $13^2$, $17^2$, $19^2$, $23^2\pmod{840}$. Similarly, Sierpi\'nski \cite{sierpinski1956decompositions} conjectured that the conclusion is true in case $m=5$. Stewart \cite[P203]{stewart1964theory} proved that Sierpi\'nski's conjecture holds for all $n\le 105743881$ and $n\not\equiv 1\pmod{278460}$.

In this paper, we mainly generalize the conclusions of Yao \cite{yao2020note} and Ding \cite{ding2019note} to $\mathbb Q$, prove that any positive irreducible fraction has infinitely many faithful decompositions consisting only of unit fractions except for one term, and show the faithful decomposition of $\frac 4n$ written as a sum of three fractions at most one non-unit fraction. The main results of this paper are as follows.

\begin{theorem}\label{thm:length_t+2}
Let $\frac mn$ be a positive irreducible fraction and $t$ be a positive integer. If $2\le t\le\frac mn<t+1$, then the length of faithful decomposition of $\frac mn$ is at least $t+2$. There are infinitely many faithful decompositions with length $t+2$.
\end{theorem}

\begin{theorem}\label{thm:all1except1term}
Let $\frac mn$ be a positive irreducible fraction and $\Omega$ be a finite set of positive integers. There are infinitely many ways to choose an integer $x$ and a sequence $\{b_i\}_{i=1}^{t+1}$ of positive integers such that
\[
\frac mn = \frac1{b_1} + \dots + \frac1{b_t} + \frac{x}{b_{t+1}} + \frac1{nb_1\cdots b_tb_{t+1}}
\]
is faithful, where $b_i$ is coprime to $\{b_j\}_{j\neq i}\cup \Omega$ for $1\le i\le t+1$.
\end{theorem}

\begin{theorem}\label{thm:ding}
Let $\frac mn$ be a positive irreducible fraction. For any integer partition $m=m_1+\dots +m_e$, $1\le m_1\le\dots\le m_e\le m$, $e\ge1$, there are infinitely many ways to choose a sequence $\{n_i\}_{i=1}^{k_e}$ of positive integers such that
\[
\frac mn = \frac{1}{n_1} + \dots + \frac{1}{n_{k_e}}.
\]
Let
\[
S_e:=\Bigg\{\sum_{i\in I}\frac{1}{n_i}\in\frac 1n\mathbb{Z}:I\subset\{1,\dots,k_e\}\Bigg\},
\quad
T_e:=\Bigg\{\sum_{i\in I}\frac{m_i}{n}:I\subset\{1,\dots,e\}\Bigg\}.
\]
Then $S_e=T_e$.
\end{theorem}

\begin{theorem}\label{thm:4/n}
Let $n\ge5$, then the irreducible proper fraction $\frac 4n$ has a faithful de\-com\-po\-si\-tion in the form of
\begin{equation}\label{equ:4/n}
\frac 4n = \frac 1x + \frac 1y + \frac rz, \qquad r=1\text{ or }2.
\end{equation}
\end{theorem}


\section{General Properties of Faithful Decomposition of Rationals}

In this section, we will prove some useful propositions about faithful decomposition.

\begin{proposition}[Divisibility]\label{prop:div_p}
If the decomposition \eqref{equ:decomposition} is faithful, then for any positive integer $C$, the decompositon
\begin{equation}\label{equ:div_p}
\frac {m}{Cn}=\frac{a_1}{Cb_1}+\frac{a_2}{Cb_2}+\cdots+\frac{a_t}{Cb_t}
\end{equation}
is also faithful.
\end{proposition}

\begin{proof}
If \Cref{equ:div_p} is not faithful, then there exist integers $x_i$ and $m'$ with $0\le x_i\le a_i$ and $0<m'<m$ for $1\le i\le t$ such that 
\[
\frac{m'}{Cn}=\frac{x_1}{Cb_1}+\frac{x_2}{Cb_2}+\cdots+\frac{x_t}{Cb_t}.
\]
Then $\frac{m'}{n}$ is decomposed as
\[
\frac{m'}{n}=\frac{x_1}{b_1}+\frac{x_2}{b_2}+\cdots+\frac{x_t}{b_t},
\]
a contradiction.
\end{proof}

\begin{proposition}\label{prop:not_div_and_less}
If the decomposition \eqref{equ:decomposition} is faithful, then $b_i\nmid n$ and $a_i<\frac{b_i}{(b_i,n)}$ for all $1\le i\le t$. 
\end{proposition}

\begin{proof}
Without loss of generality, write $d=(b_1,n)$, $b_1=db_1'$, $n=dn'$.

If $b_1\mid n$, then $n=b_1n'$. Since $\frac mn > \frac{a_1}{b_1}$, we have $m>\frac{a_1n}{b_1}=a_1n'$. We have $\frac{a_1n'}{n} = \frac{a_1}{b_1}$. Hence, the decomposition of $\frac mn$ is not faithful.

If $a_1\ge b_1'$, we have $m>\frac{a_1n}{b_1} = \frac{a_1n'}{b_1'}\ge n'$ as $\frac mn > \frac{a_1}{b_1}$. So, $\frac{n'}{n} = \frac 1d = \frac{b_1'}{b_1}$. The de\-com\-po\-si\-tion of $\frac mn$ is not faithful, either.
\end{proof}

\begin{proposition}\label{prop:faithful_in_Q}
Suppose that $\frac mn>0$ and $\frac{a_1}{b_1}$, $\frac{a_2}{b_2}$, $\dots$, $\frac{a_t}{b_t}$ are proper fractions with
\begin{equation}\label{equ:pairwise_coprime}
\frac mn=\frac{a_1}{b_1} + \frac{a_2}{b_2} + \dots + \frac{a_t}{b_t} + \frac{1}{nb_1b_2\cdots b_t}.
\end{equation}
If $n$, $b_1$, $b_2$, $\dots$, $b_t$ are pairwise coprime, then the decomposition \eqref{equ:pairwise_coprime} is faithful.
\end{proposition}

\begin{proof}
If \Cref{equ:pairwise_coprime} is not faithful, there exist integers $x_i$ and $m'$ with $0\le x_i\le a_i$ and $0<m'< m$ for $1\le i\le t$ such that 
\[
\frac{m'}{n}
= \frac{x_1}{b_1} + \frac{x_2}{b_2} + \dots + \frac{x_t}{b_t}
= \Bigg(\sum_{i=1}^t x_i\prod_{j\neq i}b_j\Bigg) \mathbin{\Bigg/} \prod_{i=1}^t b_i.
\]
Thus, $n\sum_{i=1}^t x_i\prod_{j\neq i} b_j$ is divisible by $b_i$ for $1\le i\le t$. Since $n$, $b_1$, $\dots$, $b_t$ are pairwise coprime, so we have $b_i\mid x_i$, which contradicts the fact that $\frac{a_i}{b_i}$ is a proper fraction. 
\end{proof}

\begin{proposition}\label{prop:2fenjie}
Let $\frac mn$ be a irreducible proper fraction, $x$ and $y$ be integers with $ym-xn=1$. Then the decomposition
\begin{equation}\label{equ:mn1}
\frac mn=\frac xy + \frac 1{ny},
\end{equation}
is faithful.
\end{proposition}

\begin{proof}
Since $(n,y)=1$ and $x<y$, for any $0< x'\le x$, we have $\frac{x'}{y}\notin\frac1n\mathbb{Z}$. Therefore the decomposition \eqref{equ:mn1} is faithful. 
\end{proof}

Here are two examples of faithful decomposition of rationals, also used for the proof of \Cref{thm:4/n}.

\begin{example}\label{exp:some_special_example}
Let $P$ be a perfect number. It is easy to verify the faithful decomposition as follows.
\[
1 = \sum_{\substack{d\vert P,\,d\neq1}}\frac 1d,
\qquad
\frac 49 = \frac 14 + \frac 16 + \frac 1{36}.
\]
\end{example}

\section{Proof of \Cref{thm:length_t+2}}


\begin{proof}[Proof of \Cref{thm:length_t+2}]
We first prove that there exists a faithful decomposition with length $t+2$. 

Take prime numbers $p_1<p_2<\dots<p_t$ large enough such that $\sum_{i=1}^t\frac{1}{p_i}<t+1-\frac mn$ $\big($e.g. $p_1>\frac{tn}{n-m+tn}\big)$ and $p_i\nmid n$ for $1\le i\le t$, then 
\[
0<\frac mn -t+\sum_{i=1}^t\frac{1}{p_i}<1.
\]
Then by \Cref{equ:mn1} and \Cref{prop:faithful_in_Q}, there exist $x$ and $y$ with $\big(y,n\prod_{i =1}^t p_i\big)=1$ such that 
\[
\frac mn=t-\sum_{i=1}^t\frac{1}{p_i}+\frac xy + \frac1{np_1\cdots p_ty}
=\frac{p_1-1}{p_1}+\dots+\frac{p_t-1}{p_t}+\frac xy + \frac1{np_1\cdots p_t y}
\]
is faithful with length $t+2$.

Next we prove that there is no faithful decomposition with length $t+1$.

Suppose that there exists a faithful decomposition
\begin{equation}\label{equ:n+1}
\frac mn = \frac{a_1}{b_1} + \frac{a_2}{b_2} + \dots + \frac{a_{t+1}}{b_{t+1}}
\end{equation}
with length $t+1$. By \Cref{prop:not_div_and_less}, we have $a_i<\frac{b_i}{(b_i,n)}$ for $i=1,2,\dots t+1$. Hence, we obtain $\frac{a_i}{b_i}<\frac{1}{(b_i,n)}\le\frac 12$ if $(b_i,n)>1$. 

If $b_{i_1},b_{i_2}$ are not coprime to $n$, then
\[
\frac mn = \sum_{i=1}^{t+1}\frac{a_i}{b_i} < \frac 12 + \frac 12 + (t-1) = t,
\]
which contradicts $\frac mn \ge t$. So there is at most one $(b_i,n)>1$ in $b_1,b_2,\dots,b_{t+1}$.

Without loss of generality, suppose $(b_1,n)>1$ and $(b_i,n)=1$ for $i=2,\dots,t+1$. \Cref{equ:n+1} implies 
\[
n\sum_{i=1}^{t+1}a_i\prod_{j\neq i}b_j = m\prod_{i=1}^{t+1}b_i,
\]
where $n$ is coprime to $m,b_2,\dots,b_{t+1}$, so $n\mid b_1$. Hence, $\frac{a_1}{b_1}<\frac 1{(b_1,n)}=\frac 1n$,  and then
\[
t \le \frac mn = \frac{a_1}{b_1} + \sum_{i=2}^{t+1}\frac{a_i}{b_i} < \frac 1n + t.
\]
So $\frac mn = t\in\mathbb Z$. According to Yao \cite[Theorem 2]{yao2020note}, the integer $t$ has no faithful de\-com\-po\-si\-tion with length $t+1$. A contradiction.

In summary, the length of faithful decomposition of $\frac mn$ is at least $t+2$.
\end{proof}

\section{Proof of \Cref{thm:all1except1term,thm:ding}}


\begin{proof}[Proof of \Cref{thm:all1except1term}]
Since $\sum_{p\text{ is prime}}\frac{1}{p}$ diverges \cite[P21]{GTM84}, we have infinitely many ways to choose enough positive integers $b_1, b_2, \dots, b_t$ such that
\[
0<\frac mn - \frac1{b_1}-\dots-\frac1{b_t}<1,
\]
where $b_i$ is coprime to $\{b_j\}_{j\neq i}\cup \Omega$ for $1\le i\le t$. Write
\[
\frac mn = \frac1{b_1}+\dots+\frac1{b_t} + \frac{z}{nb_1\cdots b_t}.
\]
Since $z<n\prod_{i=1}^t b_i$ and $\big(z,n\prod_{i=1}^t b_i\big)=1$, we can choose positive integers $x_0$, $y_0$ such that $zy_0-x_0\big(n\prod_{i=1}^t b_i\big)=1$. This equation remains true when we replace $y_0$ and $x_0$ by $y_0+a\big(n\prod_{i=1}^t b_i\big)$ and $x_0+az$ respectively for any positive integer $a$. By Dirichlet prime number theorem on arithmetic progressions \cite[Theorem 15]{hardy2008}, we have infinite ways to choose an integer $a_0$ such that $b_{t+1}:=y_0+a_0\big(n\prod_{i=1}^t b_i\big)$ is coprime to $\{b_j\}_{j=1}^t\cup \Omega$. By \Cref{prop:faithful_in_Q}, the decomposition
\[
\frac mn = \frac1{b_1} + \dots + \frac1{b_t} + \frac{x_0+a_0z}{b_{t+1}} + \frac1{nb_1\cdots b_tb_{t+1}}
\]
is faithful.
\end{proof}

Using a similar proof procedure, we can obtain the following general proposition.

\begin{proposition}\label{prop:general_pairwise_coprime}
Let $\frac mn$ be a positive irreducible fraction and $\Omega$ be a finite set of positive integers. There are infinitely many ways to choose a sequence $\big\{\frac{a_i}{b_i}\big\}_{i=1}^t$ of proper fractions such that
\[
\frac mn = \frac{a_1}{b_1} + \dots + \frac{a_t}{b_t} + \frac1{nb_1\cdots b_t}
\]
is faithful, where $b_i$ is coprime to $\{b_j\}_{j\neq i}\cup \Omega$ for $1\le i\le t$.
\end{proposition}

\begin{proof}
We can choose the sequence $\big\{\frac{a_i}{b_i}\big\}_{i=1}^{t-1}$ of proper fractions such that
\[
0<\frac mn - \frac{a_1}{b_1}-\dots-\frac{a_{t-1}}{b_{t-1}}<1,
\]
where $b_i$ is coprime to $\{b_j\}_{j\neq i}\cup \Omega$ for $1\le i\le t-1$. By Dirichlet's theorem, there exist integers $a_t$ and $b_t$ with $a_t<b_t$ such that
\begin{equation}\label{equ:general_pairwise_coprime}
\frac mn = \frac{a_1}{b_1} + \dots + \frac{a_{t-1}}{b_{t-1}} + \frac{a_t}{b_t} + \frac1{nb_1\cdots b_t},
\end{equation}
where $b_t$ is coprime to $\{b_j\}_{j=1}^{t-1}\cup \Omega$. By \Cref{prop:faithful_in_Q}, the decomposition \eqref{equ:general_pairwise_coprime} is faithful.
\end{proof}

\begin{proof}[Proof of \Cref{thm:ding}]
According to \Cref{prop:general_pairwise_coprime}, for $i=1,\dots,e$, let
\[
\frac{m_i}{n}=\frac{a_{i1}}{b_{i1}}+\dots + \frac{a_{it_i}}{b_{it_i}} + \frac{1}{n b_{i1}\cdots b_{it_i}}
\]
be a faithful decomposition of $\frac{m_i}{n}$, where $b_{ij}$ is coprime to $\{b_{il}\}_{l\neq j}\cup\big(\{n\}\cup\bigcup_{k\neq i}\{b_{kl}\}_{l=1}^{t_k}\big)$ for $1\le j\le t_i$. In other words, $\{\{b_{ij}\}_{j=1}^{t_i}\}_{i=1}^e$ are pairwise coprime and $n$ is coprime to $\prod_{i=1}^e\prod_{j=1}^{t_i} b_{ij}$. Write
\begin{equation}\label{equ:m/n=sum1/ni}
\frac mn = \sum_{i=1}^e\frac{m_i}{n}=\sum_{i=1}^e\Bigg(
\frac{a_{i1}}{b_{i1}}+\dots + \frac{a_{it_i}}{b_{it_i}} + \frac{1}{n b_{i1}\cdots b_{it_i}}
\Bigg)
=:\sum_{i=1}^{k_e}\frac1{n_i}.
\end{equation}
It is obvious that $S_e\supset T_e$.

We first consider the case $e=1$. Because
\[
\frac mn = \frac{a_{11}}{b_{11}} + \dots + \frac{a_{1t_1}}{b_{1t_1}} + \frac1{nb_{11}\cdots b_{1t_1}}
\]
is faithful, the partial sums $\sum_{i\in I}\frac{1}{n_i}$ lying in $\frac 1n\mathbb{Z}$ must be equal to $0$ or $\frac mn$, where $I\subset\{1,\dots,k_1\}$. Hence, $S_1\subset T_1$.

Suppose that $S_e\subset T_e$ holds for $e\le d-1$, where $d\ge2$ is an integer. Now consider the case $e=d$. 

According to \Cref{equ:m/n=sum1/ni}, choose the partial sums of the form
\[
\sum_{i=1}^{d}\Bigg(
\frac{c_{i1}}{b_{i1}}+\dots + \frac{c_{it_i}}{b_{it_i}} + \frac{\delta_i}{nb_{i1}\cdots b_{it_i}}
\Bigg)\in S_d,
\]
where $\delta_i=0$ or $1$, and $0\le c_{ij}\le a_{ij}$ for $1\le j\le t_i$. Let
\[
A=\sum_{i=1}^{d-1}\Bigg(
\sum_{j=1}^{t_i}\frac{c_{ij}}{b_{ij}} + \frac{\delta_i}{nb_{i1}\cdots b_{it_i}}
\Bigg),
\qquad
B=\sum_{j=1}^{t_d}\frac{c_{dj}}{b_{dj}} + \frac{\delta_d}{nb_{d1}\cdots b_{dt_d}},
\]
so that $A+B\in\frac 1n\mathbb Z$. Since both $A\prod_{i=1}^{d-1}\prod_{j=1}^{t_i}b_{ij}$ and $(A+B)\prod_{i=1}^{d-1}\prod_{j=1}^{t_i}b_{ij}$ lie in $\frac 1n\mathbb Z$, so does $B\prod_{i=1}^{d-1}\prod_{j=1}^{t_i}b_{ij}$. But $n\prod_{j=1}^{t_d} b_{dj}$ is coprime to $\prod_{i=1}^{d-1}\prod_{j=1}^{t_i}b_{ij}$, so $B\in\frac 1n\mathbb Z$. Since $A+B\in\frac 1n\mathbb Z$, so is $A$. Hence, $A\in S_{d-1}$.

According to the inductive assumption, we have $A\in T_{d-1}$. Since the decomposition of $\frac{m_d}{n}$ is faithful, then $B=0$ or $B=\frac{m_d}{n}$. Hence, we have $A+B\in T_d$ and $S_d\subset T_d$.

Therefore, we immediately obtain $S_e=T_e$. This completes the proof.
\end{proof}

\section{Proof of \Cref{thm:4/n}}

In this section, we first give two conclusions about faithful decomposition with three fractions, and use them to prove \Cref{thm:4/n}.

\begin{proposition}\label{thm:r/s_1n}
Suppose $(y,y_2)=1$, then
\begin{equation}\label{equ:r/s_1n}
\frac mn = \frac{1}{y_2} + \frac 1{y_1} + \frac{x}{yn}
\end{equation}
is faithful if and only if $x<y$ and $n\neq m'y_2$ for any positive integer $m' < m$.
\end{proposition}

\begin{proof}
The direct implication is immediately obtained by \Cref{prop:not_div_and_less}.

For the converse implication, if \Cref{equ:r/s_1n} does not give a faithful decomposition of $\frac mn$, then there exist a postive integer $m'<m$ and an integer $x'$ such that 

\textsc{Case} 1. 
\[
\frac{m'}{n} = \frac{x'}{yn}, 
\qquad 1\le x'\le x.
\]
So $m'=\frac{x'}{y}\in\mathbb Z$, thus $y\le x'\le x$. A contradiction.

\textsc{Case} 2.
\[
\frac{m'}{n} = \frac{1}{y_2} + \frac{x'}{yn},
\qquad 0\le x'\le x.
\]
So $y_2(m'y-x')=yn$. Because $(y,y_2)=1$, we have $y_2\mid n$ and then $\frac{x'}{y}=m'-\frac{n}{y_2}\in \mathbb Z$. Therefore, we obtain either $y\le x'\le x$, or $x'=0$ which implies $n=m'y_2$, a contradiction. 
\end{proof}

\begin{proposition}\label{cor:m/n_faithful_iff}
Let $m\ge3$, $\frac mn$ be a irreducible proper fraction. Take $r\equiv -2n\pmod m$ with $0< r < m$. 
\begin{enumerate}[(1)]
\item When $2n+r\equiv m\pmod{2m}$, the decomposition
\begin{equation}\label{equ:equiv_m_mod_2m}
\frac mn = \frac{1}{\frac{2n+r+m}{2m}} + \frac{1}{\frac{2n+r+m}{2m}\frac{2n+r}{m}} + \frac{r}{\frac{2n+r}{m}n}
\end{equation}
is faithful if and only if $n>\frac{r(m-1)}{2}$ and $n\neq \frac{m'(r+m )}{2(m-m')}$ for any $\frac m2 < m' < m$.  \label{enu:equiv_m_mod_2m}

\item When $2n+r\equiv 0\pmod{2m}$, the decomposition
\begin{equation}\label{equ:equiv_0_mod_2m}
\frac mn = \frac{1}{\frac{2n+r+2m}{2m}} + \frac{1}{\frac{2n+r+2m}{2m}\frac{2n+r}{2m}} + \frac{r/2}{\frac{2n+r}{2m}n}
\end{equation}
is faithful if and only if $n>\frac{r(m-1)}{2}$ and $n\neq \frac{m'(r+2m)}{2(m-m')}$ for any $\frac 25m < m' < m$. \label{enu:equiv_0_mod_2m}
\end{enumerate}
\end{proposition}

\begin{proof}
It is obvious that $r<\frac{2n+r}{m}$, $\frac r2<\frac{2n+r}{2m}$ and $n>\frac{r(m-1)}{2}$ are equivalent.

(\ref{enu:equiv_m_mod_2m}) It is easy to check that $\big(\frac{2n+r+m}{2m},\frac{2n+r}{m}\big)=1$ in \Cref{equ:equiv_m_mod_2m}.
If there is $0<m'<m$ such that $n=m'\frac{2n+r+m}{2m}$, then $2(m-m')n=m'(r+m)$. Notice that $r<m<n$, so $2(m-m')m < 2(m-m')n = m'(r+m) < 2mm'$, thus $\frac m2 < m' < m$. 

Therefore, \Cref{equ:equiv_m_mod_2m} is a faithful decomposition by \Cref{thm:r/s_1n}. 

(\ref{enu:equiv_0_mod_2m}) It is easy to check that $2\mid r$ and $\big(\frac{2n+r+2m}{2m},\frac{2n+r}{2m}\big)=1$ in \Cref{equ:equiv_0_mod_2m}.
If there is $0<m'<m$ such that $n=m'\frac{2n+r+2m}{2m}$, then $\frac 25 m < m' < m$. 

Therefore, \Cref{equ:equiv_m_mod_2m} is a faithful decomposition by \Cref{thm:r/s_1n}. 
\end{proof}

Finally, we give a proof of \Cref{thm:4/n}.

\begin{proof}[Proof of \Cref{thm:4/n}]
Because $n>4$ is odd, so $-2n\equiv2\pmod 4$, and then $r=2$, satisfying $n>\frac{2(4-1)}{2}=3$. 

\textsc{Case} 1. If $n\equiv 1\pmod{4}$, then $2n+2\equiv4\pmod8$. By \Cref{equ:equiv_m_mod_2m} and \Cref{cor:m/n_faithful_iff} (\ref{enu:equiv_m_mod_2m}),
\[
\frac 4n = \frac{1}{\frac{n+3}{4}} + \frac{1}{\frac{n+3}{4}\frac{n+1}{2}} + \frac{2}{\frac{n+1}{2}n}
\]
is faithful when $n\neq\frac{3(2+4)}{2(4-3)}=9$. The faithful decomposition of $\frac 49$ is given in \Cref{exp:some_special_example}. 

\textsc{Case} 2. If $n\equiv 3\pmod{4}$, then $2n+2\equiv0\pmod8$. By \Cref{equ:equiv_0_mod_2m} and \Cref{cor:m/n_faithful_iff} (\ref{enu:equiv_0_mod_2m}),
\[
\frac 4n = \frac{1}{\frac{n+5}{4}} + \frac{1}{\frac{n+5}{4}\frac{n+1}{4}} + \frac{1}{\frac{n+1}{4}n}
\]
is faithful when $n\neq\frac{2(2+8)}{2(4-2)}=5$ and $n\neq\frac{3(2+8)}{2(4-3)}=15$. Because the faithful decom\-po\-si\-tion of $\frac 45$ is given in \textsc{Case} 1, $\frac{4}{15}$ also has a faithful decomposition form \Cref{equ:4/n} by \Cref{prop:div_p}.
\end{proof}

\begin{remark}
Using \Cref{cor:m/n_faithful_iff}, the proof process of \Cref{thm:4/n} can be used to prove the cases of $\frac 3n$ and $\frac 5n$.
\end{remark}

\section*{Acknowledgments}

We would like to thank Yuqing He, and all the anonymous referees for helpful comments. Pingzhi Yuan was supported by the National Natural Science Foundation of China (Grant No. 12171163) and the Basic and Applied Basic Research Foundation of Guangdong Province (Grant No. 2024A1515010589). Hongjian Li was supported by the Project of Guangdong University of Foreign Studies (Grant No. 2024RC063).

\bibliographystyle{plain}
\bibliography{refs}
\end{document}